\documentclass{amsart}
\title[Category $\mathcal{J}$ Modules for Divergence Zero Vector Fields]{Classification of Category $\mathcal{J}$ Modules for Divergence Zero Vector Fields on a Torus}
\author{Yuly Billig and John Talboom}
\newtheorem{thm}{Theorem}[section]

\newtheorem{lem}[thm]{Lemma}
\newtheorem{prop}[thm]{Proposition}
\newtheorem{defn}[thm]{Definition}

\newtheorem*{theorem*}{Theorem}
\begin{document}
\begin{abstract}
We consider a category of modules that admit compatible actions of the commutative algebra of Laurent polynomials and the Lie algebra of divergence zero vector fields on a torus and have a weight decomposition with finite dimensional weight spaces. We classify indecomposable and irreducible modules in this category.
%
\end{abstract}
\maketitle

\section{Introduction}
Consider the algebra $A_N=\mathbb{C}[t_1^{\pm1},\dots,t_{N}^{\pm1}]$ and Lie algebra $\text{Der}(A_N)$ of derivations of $A_N$. 
 The Lie algebra $\text{Der}(A_N)$ may be identified with the Lie algebra of polynomial vector fields on an $N$-dimensional torus (see Section 2).
 In \cite{R2} Eswara Rao considered modules that admit compatible actions of both the Lie algebra
$\text{Der}(A_N)$ and the commutative algebra $A_N$. We refer to such modules as $(A_N,\text{Der}(A_N))$-modules.
Tensor fields on a torus provide examples of modules in this class.
Eswara Rao classified in \cite{R2} all irreducible $(A_N,\text{Der}(A_N))$-modules with finite-dimensional weight spaces
and proved that all such modules are in fact tensor modules. This result was extended in \cite{B} to a classification of indecomposable modules in this category. To accomplish this it was shown that the action of the Lie algebra is polynomial 
(see \cite{BB} and \cite{BZ}); a strategy which will be used in the current paper.

In \cite{R1} Eswara Rao determines conditions for irreducibility of tensor modules for $\text{Der}(A_N)$. 
Restrictions of these tensor modules to the subalgebra of divergence zero vector fields, denoted here by $\mathcal{S}_N$, are studied in \cite{T}, and it was found that these modules remain irreducible under similar conditions.
The goal of this paper is to study the category $\mathcal{J}$ of $(A_N, \mathcal{S}_N)$-modules with finite-dimensional weight spaces and classify irreducible and indecomposable modules in this category. 

Let $S_N^+$ be the Lie algebra of divergence zero elements of $\text{Der}(\mathbb{C}[x_1,\dots,x_N])$ of non-negative degrees, $\mathcal{H}$ the three dimensional Heisenberg algebra and $\mathfrak{a}_N$ an abelian algebra of dimension $N$. The main result of this paper is Theorem \ref{classification} which is stated below (the action of $\mathcal{S}_N$ will also be given).

\begin{theorem*}
Let $\lambda\in\mathbb{C}^N$ and let $\mathcal{J}_{\lambda}$ be the subcategory of modules in category $\mathcal{J}$ supported on $\lambda+\mathbb{Z}^N$. 
\begin{enumerate}
\item[(a)]
For $N=2$ there is an equivalence of categories between the category of finite dimensional modules for $S_2^+\oplus\mathcal{H}$ and $\mathcal{J}_{\lambda}$. This equivalence maps $U$ to $A_2\otimes U$ where $U$ is a finite dimensional module for $S_2^+\oplus\mathcal{H}$. 
\item[(b)]
For $N\geq 3$, there is an equivalence of categories between the category of finite dimensional modules for $S_N^+\oplus\mathfrak{a}_N$ and $\mathcal{J}_{\lambda}$. This equivalence maps $U$ to $A_N\otimes U$ where $U$ is a finite dimensional module for $S_N^+\oplus\mathfrak{a}_N$.
\end{enumerate}
\end{theorem*}

Section 2 of the current paper will introduce category $\mathcal{J}$ for the Lie algebra of divergence zero vector fields on an $N$-dimensional torus. An immediate consequence of this definition is that weights for an indecomposable module $J$ in category $\mathcal{J}$ form a single coset $\lambda+\mathbb{Z}^N$ where $\lambda\in\mathbb{C}^N$. Furthermore, since $J$ is a free $A_N$-module of finite rank, all weight spaces have the same dimension and it follows that $J\cong A_N\otimes U$, where $U$ is any weight space of $J$.

Section 3 contains the bulk of the proof for the classification of category $\mathcal{J}$. It begins by showing that the action of $\mathcal{S}_N$ on $J$ is completely determined by the action of a certain Lie algebra on $U$. The remainder of the section is to show that $\mathcal{S}_N$ acts on $J$ by certain $\text{End}(U)$-valued polynomials. It is seen however that the case $N=2$ is exceptional in this regard and it must be considered separately from the cases where $N\geq 3$.

The main results are presented in Sections 4 and 5. Here the so-called polynomial action on $U$ is seen to be a representation of the Lie algebra $S_N^+$, along with the three dimensional Heisenberg in the case $N=2$, or an abelian algebra in the case $N\geq3$. In the case of irreducible modules the action of $S_N^+$ simplifies to a representation of $\mathfrak{sl}_N$, the Lie algebra of $N\times N$ matrices with trace zero over $\mathbb{C}$, and the three dimensional Heisenberg converts to an abelian algebra when $N=2$. 

Irreducible representations for the case $N=2$ are studied in \cite{JL} by Jiang and Lin. Lemma \ref{eigenvector} below makes use of a technique found in this paper in order to obtain a family of eigenvectors. This provides a crucial step in the classification.
\section{Preliminaries}
Let $A_N=\mathbb{C}[t_1^{\pm1},\dots,t_{N}^{\pm1}]$ be the algebra of Laurent polynomials over $\mathbb{C}$. Elements of $A_N$ are presented with multi-index notation $t^r=t_1^{r_1}\dots t_{N}^{r_{N}}$ where $r=(r_1,\dots,r_{N})\in\mathbb{Z}^{N}$. Let $\{e_1,\dots,e_N\}$ denote the standard basis for $\mathbb{Z}^N$. For $k\in\mathbb{Z}^N	$, $|k|=k_1+\dots+k_N$, $k!=k_1!\dots k_N!$ and $\binom{r}{k}=\frac{r!}{k!(r-k)!}$. Denote the set of non-negative integers by $\mathbb{Z}_{\geq0}$.

For $i\in\{1,\dots,N\}$, let $d_i=t_i\frac{\partial}{\partial t_i}$. The vector space of derivations of $A_N$, $\text{Der}(A_N)=\text{Span}_{\mathbb{C}}\left\{t^rd_i|i\in\{1,\dots,N\}, r\in\mathbb{Z}^{N}\right\}$, forms a Lie algebra called the Witt algebra denoted here by $\mathcal{W}_N$. The Lie bracket in $\mathcal{W}_N$ is given by $[t^rd_i,t^sd_j]=s_it^{r+s}d_j-r_jt^{r+s}d_i$.

Geometrically, $\mathcal{W}_N$ may be interpreted as the Lie algebra of (complex-valued) polynomial vector fields on an $N$ dimensional torus via the mapping $t_j=e^{\sqrt{-1}\theta_j}$ for all $j\in\{1,\dots,N\}$, where $\theta_j$ is the $j$th angular coordinate. This has an interesting subalgebra, the Lie algebra of divergence-zero vector fields, denoted $\mathcal{S}_N$.

The change of coordinates $t_j=e^{\sqrt{-1}\theta_j}$, gives $\frac{\partial}{\partial \theta_j}=\frac{\partial t_j}{\partial \theta_j}\cdot\frac{\partial}{\partial t_j}=\sqrt{-1} t_j\frac{\partial}{\partial t_j}=\sqrt{-1} d_j$. Thus an element $X=\sum_{j=1}^{N}f_j(t)d_j\in \mathcal{W}_N$ can be written in the form $X=-\sqrt{-1}\sum_{j=1}^{N}f_j(t)\frac{\partial}{\partial \theta_j}$. The divergence of $X$ with respect to the natural volume form in angular coordinates is then $-\sqrt{-1}\sum_{j=1}^{N}\frac{\partial f_j}{\partial \theta_j}=\sum_{j=1}^{N}t_j\frac{\partial f_j}{\partial t_j}$. Letting $d_{ab}(r)=r_{b}t^rd_a-r_at^rd_{b}$, it follows that 
\begin{equation*}
\mathcal{S}_N=\text{Span}_{\mathbb{C}}\left\{d_a,d_{ab}(r)|a,b\in\{1,\dots,N\}, r\in\mathbb{Z}^{N}\right\}
\end{equation*}
and has commutative Cartan subalgebra $\mathfrak{h}=\text{Span}_{\mathbb{C}}\{d_j|j\in\{1,\dots,N\}\}$. It will be useful to have the Lie bracket of $\mathcal{S}_N$ in terms of the elements $ d_{ab}(r)$. For $r,s\in\mathbb{Z}^N$ and $a,b,p,q\in\{1,\dots,N\}$, $[d_{a},d_{pq}(r)]=r_ad_{pq}(r),$ and
\begin{multline*}
[d_{ab}(r),d_{pq}(s)]\\=r_bs_pd_{aq}(r+s)-r_bs_qd_{ap}(r+s)-r_as_pd_{bq}(r+s)+r_as_qd_{bp}(r+s).
\end{multline*}
By definition $d_{ab}(0)=0,d_{aa}(r)=0$ and $d_{ba}(r)=-d_{ab}(r)$. When $N\geq3$, $d_{ab}(r)=0$ in the case $r_a=r_b=0$, and in general,
\begin{equation*}
r_pd_{ab}(r)+r_ad_{bp}(r)+r_bd_{pa}(r)=0.
\end{equation*}

A family of modules for $\mathcal{W}_N$ called category $\mathcal{J}$ was defined in \cite{B}. An analogous category of modules for $\mathcal{S}_N$ is defined as follows:
\begin{defn}\label{CatJ}
Let $N>1$. An $\mathcal{S}_N$-module $J$ belongs to category $\mathcal{J}$ if the following properties hold:
\begin{enumerate}
\item[(J1)] The action of $d_a$ on $J$ is diagonalizable for all $a\in\{1,\dots,N\}$.
\item[(J2)] Module $J$ is a free $A_N$-module of finite rank.
\item[(J3)] For any $X\in\mathcal{S}_N,f\in A_N$ and $u\in J$, $X(fu)=(X(f))u+f(Xu).$
\end{enumerate}
\end{defn}
A submodule of any $J\in\mathcal{J}$ must be invariant under the actions of both $A_N$ and $\mathcal{S}_N$. Classifying the modules of category $\mathcal{J}$ is the goal of this paper. From property (J2) it follows that any module in $\mathcal{J}$ is a finite direct sum of indecomposable modules and hence it suffices to examine indecomposable  modules $J\in\mathcal{J}$. Using (J1) we may consider the $\mathfrak{h}$-weight decomposition, $J=\bigoplus_{\lambda\in\mathbb{C}^N}J_{\lambda}$ where $J_{\lambda}=\{u\in J|d_a(u)=\lambda_au\}$. For $u\in J_{\lambda}$,
\begin{equation*}
d_a(d_{bc}(r)u)=d_{bc}(r)d_au+[d_a,d_{bc}(r)]u=(\lambda_a+r_a)d_{bc}(r)u
\end{equation*}
and thus $d_{bc}(r)J_{\lambda}\subset J_{\lambda+r}$. Similarly by (J3) $t^rJ_{\lambda}\subset J_{\lambda+r}$. These two relations partition the weights of $J$ into $\mathbb{Z}^N$-cosets of $\mathbb{C}^N$, and decompose $J$ into a direct sum of submodules, each corresponding to a distinct coset. Thus if $J$ is indecomposable its set of weights is one such coset $\lambda+\mathbb{Z}^N$ for $\lambda\in\mathbb{C}^N$ and $J=\bigoplus_{r\in\mathbb{Z}^N}J_{\lambda+r}$. For a fixed $\lambda\in\mathbb{C}^N$, $\mathcal{J}_{\lambda}$ shall denote the subcategory of $\mathcal{J}$ supported on $\lambda+\mathbb{Z}^N$ (i.e. the subcategory of $\mathcal{J}$ consisting of all $\mathcal{S}_N$-modules whose weights are of the form $\lambda+\mathbb{Z}^N$). From now on assume $J\in\mathcal{J}_{\lambda}$.

Let $U=J_{\lambda}$. The invertible map $t^r:U\rightarrow J_{\lambda+r}$ identifies all weight spaces with $U$ and since $J$ is a free module for the associative algebra $A_N$ it follows that any basis for $U$ is also basis for $J$ viewed as a free $A_N$-module. Furthermore the finite rank condition of property (J2) implies that $U$ must be finite dimensional. This yields that $J\cong A_N\otimes U$. Homogeneous elements of $J$ will be denoted $t^s\otimes v$, for $s\in\mathbb{Z}^N,v\in U$.
\section{Polynomial Action}
The map $d_{ab}(r):1\otimes U\rightarrow t^r\otimes U$ induces an endomorphism $D_{ab}(r):U\rightarrow U$ defined by
\begin{equation*}
D_{ab}(r)u=(t^{-r}\circ d_{ab}(r))u
\end{equation*}
for $u\in U$. Combining this with (J3) yields
\begin{equation}\label{SNactiononJ}
d_{ab}(r)(t^s\otimes v)=(r_bs_a-s_br_a)t^{r+s}\otimes v + t^{r+s}\otimes D_{ab}(r)v,
\end{equation}
and so the action of $d_{ab}(r)$ on $J$ is determined by that of $D_{ab}(r)$ on $U$.

The key to proving the main result is to show that $D_{ab}(r)$ acts on $U$ by an $\text{End}(U)$-valued polynomial in $r$ when $N\geq 3$. In this case it is said that $D_{ab}(r)$ has \emph{polynomial action}. That is, 
\begin{equation*}
D_{ab}(r)=\sum_{k\in\mathbb{Z}_{\geq0}^N}\frac{r^k}{k!}P^{(k)}_{ab},
\end{equation*}
where the $P^{(k)}_{ab}\in\text{End}(U)$ do not depend on $r$, and the sum is finite. The factor of $k!$ is there for convenience. In the case that $N=2$ a slight modification needs to be made and the corresponding expansion has the form
\begin{equation*}
D_{ab}(r)=\sum_{k\in\mathbb{Z}_{\geq0}^2}\frac{r^k}{k!}P^{(k)}_{ab}-\delta_{r,0}P^{(0)}_{ab},
\end{equation*}
where $\delta_{r,0}$ is the Kronecker delta.

Since $D_{ab}(r)=t^{-r}d_{ab}(r)$, the Lie bracket for $D_{ab}(r)$ follows from that of $d_{ab}(r)$.
\begin{align*}
&\quad\;[D_{ab}(r),D_{cd}(s)]\\
&=[t^{-r}d_{ab}(r),t^{-s}d_{cd}(s)]\\
&=t^{-r}(d_{ab}(r)(t^{-s}))d_{cd}(s)-t^{-s}(d_{cd}(s)(t^{-r}))d_{ab}(r)+t^{-r-s}[d_{ab}(r),d_{cd}(s)]\\
&=t^{-r}(-r_bs_a+r_as_b)t^{r-s}d_{cd}(s)-t^{-s}(-r_cs_d+r_ds_c)t^{s-r}d_{ab}(r)\\
&\quad+t^{-r-s}\left(r_bs_cd_{ad}(r+s)-r_bs_dd_{ac}(r+s)-r_as_cd_{bd}(r+s)+r_as_dd_{bc}(r+s)\right)\\
&=(r_as_b-r_bs_a)D_{cd}(s)+(r_cs_d-r_ds_c)D_{ab}(r)\\
&\quad+r_bs_cD_{ad}(r+s)-r_bs_dD_{ac}(r+s)-r_as_cD_{bd}(r+s)+r_as_dD_{bc}(r+s).
\end{align*}
This has special case
\begin{equation}
\label{bracket}
[D_{ab}(r),D_{ab}(s)]=(r_as_b-r_bs_a)(D_{ab}(r)+D_{ab}(s)-D_{ab}(r+s)).
\end{equation}
Note that for $N\geq3$
\begin{equation*}
r_cD_{ab}(r)+r_aD_{bc}(r)+r_bD_{ca}(r)=0.
\end{equation*}

For a function $f$ whose domain is $\mathbb{Z}^N$ the \emph{difference derivative} in direction $r\in\mathbb{Z}^N$, denoted by $\partial_rf$, is defined as
\begin{equation*}
\partial_rf(s)=f(s+r)-f(s).
\end{equation*}
Higher order derivatives are obtained by iteration and thus
\begin{equation}\label{higherorderderiv}
\partial_r^mf(s)=\sum_{i=0}^m(-1)^{m-i}\binom{m}{i}f(s+ir).
\end{equation}
To simplify notation let $\partial_a=\partial_{e_a}$. Applying the above twice yields
\begin{equation}\label{mixedpartialderiv}
\partial_{a}^m\partial_{b}^nf(s)=\sum_{i=0}^m\sum_{j=0}^n(-1)^{m+n-i-j}\binom{m}{i}\binom{n}{j}f(s+ie_a+je_b).
\end{equation}

A technique for finding eigenvectors was found in \cite{JL} and provides a key step to proving the result here (cf. Lemma 4 in \cite{JL}).
\begin{lem}\label{eigenvector}
Let $m,n\in\mathbb{Z}_{\geq0}$, and $a,b\in\{1,\dots,N\}$. Then for $m\geq1$
\begin{equation*}
[D_{ab}(-e_b),[D_{ab}(-e_a),\partial_{a}^m\partial_{b}^nD_{ab}(e_a)]]=-n(m+1)\partial_{a}^m\partial_{b}^nD_{ab}(e_a).
\end{equation*}
\end{lem}
\begin{proof}
First apply the formula for the difference derivatives,
\begin{align*}
&[D_{ab}(-e_b),[D_{ab}(-e_a),\partial_{a}^m\partial_{b}^nD_{ab}(e_a)]\\
&=\sum_{i=0}^m\sum_{j=0}^n(-1)^{m+n-i-j}\binom{m}{i}\binom{n}{j}[D_{ab}(-e_b),[D_{ab}(-e_a),D_{ab}((i+1)e_a+je_b)],
\end{align*}
then evaluate the Lie bracket for $D_{ab}(r)$,
\begin{align*}
&=-\sum_{i=0}^m\sum_{j=0}^n(-1)^{m+n-i-j}\binom{m}{i}\binom{n}{j}j[D_{ab}(-e_b),D_{ab}(-e_a)]\\
&\quad-\sum_{i=0}^m\sum_{j=0}^n(-1)^{m+n-i-j}\binom{m}{i}\binom{n}{j}j[D_{ab}(-e_b),D_{ab}((i+1)e_a+je_b)]\\
&\quad+\sum_{i=0}^m\sum_{j=0}^n(-1)^{m+n-i-j}\binom{m}{i}\binom{n}{j}j[D_{ab}(-e_b),D_{ab}(ie_a+je_b))].
\end{align*}
Simplifying the binomial coefficients and evaluating the Lie bracket yields,
\begin{align*}
&=-n\sum_{j=1}^n(-1)^{n-j}\binom{n-1}{j-1}\left(\sum_{i=0}^m(-1)^{m-i}\binom{m}{i}\right)[D_{ab}(-e_b),D_{ab}(-e_a)]\\
&\quad-n\sum_{i=0}^m\sum_{j=1}^n(-1)^{m+n-i-j}\binom{m}{i}\binom{n-1}{j-1}\\
&\quad\times(i+1)(D_{ab}(-e_b)+D_{ab}((i+1)e_a+je_b)-D_{ab}((i+1)e_a+(j-1)e_b))\\
&\quad+n\sum_{i=0}^m\sum_{j=1}^n(-1)^{m+n-i-j}\binom{m}{i}\binom{n-1}{j-1}\\
&\quad\times i(D_{ab}(-e_b)+D_{ab}(ie_a+je_b)-D_{ab}(ie_a+(j-1)e_b)).
\end{align*}
The first term vanishes because the sum in parentheses is zero for $m\geq1$. For a similar reason the terms involving  $D_{ab}(-e_b)$ will vanish. A change of summation index causes a sign change leaving,
\begin{align*}
&n\sum_{i=0}^m\sum_{j=0}^{n-1}(-1)^{m+n-i-j}\binom{m}{i}\binom{n-1}{j}\\
&\times(D_{ab}((i+1)e_a+(j+1)e_b)-D_{ab}((i+1)e_a+je_b))\\
&-mn\sum_{i=0}^{m-1}\sum_{j=0}^{n-1}(-1)^{m+n-i-j}\binom{m-1}{i}\binom{n-1}{j}\\
&\times(D_{ab}((i+2)e_a+(j+1)e_b)-D_{ab}((i+2)e_a+je_b))\\
&+mn\sum_{i=0}^{m-1}\sum_{j=0}^{n-1}(-1)^{m+n-i-j}\binom{m-1}{i}\binom{n-1}{j}\\
&\times(D_{ab}((i+1)e_a+(j+1)e_b)-D_{ab}((i+1)e_a+je_b)).
\end{align*}
Applying the definition of the difference derivative combines terms to give
\begin{align*}
&\quad n\sum_{i=0}^{m}\sum_{j=0}^{n-1}(-1)^{m+n-i-j}\binom{m}{i}\binom{n-1}{j}\partial_{b}D_{ab}((i+1)e_a+je_b)\\
&\quad-mn\sum_{i=0}^{m-1}\sum_{j=0}^{n-1}(-1)^{m+n-i-j}\binom{m-1}{i}\binom{n-1}{j}\\
&\quad\times(\partial_{b}D_{ab}((i+2)e_a+je_b)-\partial_{b}D_{ab}((i+1)e_a+je_b))\\
&=-n\partial_{a}^m\partial_{b}^{n-1}\left(\partial_{b}D_{ab}(e_a)\right)\\
&\quad-mn\sum_{i=0}^{m-1}\sum_{j=0}^{n-1}(-1)^{m+n-i-j}\binom{m-1}{i}\binom{n-1}{j}\partial_{a}\partial_{b}D_{ab}((i+1)e_a+je_b)\\
&=-n\partial_{a}^m\partial_{b}^nD_{ab}(e_a)-mn\partial_{a}^{m-1}\partial_{b}^{n-1}\left(\partial_{a}\partial_{b}D_{ab}(e_a)\right)\\
&=-n(m+1)\partial_{a}^m\partial_{b}^nD_{ab}(e_a).
\end{align*}
\end{proof}
The lemma above shows that for various $m$ and $n$, $\partial_{a}^m\partial_{b}^nD_{ab}(e_a)$ are eigenvectors for $\text{ad}(D_{ab}(-e_b))\text{ad}(D_{ab}(-e_a))$ and yield infinitely many distinct eigenvalues. This fact will be used to show that $\partial_{a}^m\partial_{b}^nD_{ab}(e_a)=0$ for large enough values of $m$ and $n$.

The following lemma was proven in \cite{B} and is presented here without proof.
\begin{lem}[\cite{B} Lemma 3.4]\label{finiteeigenvals}
Let $\mathfrak{L}$ be a Lie algebra with nonzero elements $y, y_1$, $y_2,\dots$ with the property that
\[[y,y_i]=\alpha_iy_i\]
for $i=1,2,\dots$, and $\alpha_i\in\mathbb{C}$. Then for a finite dimensional representation $(U,\rho)$ of $\mathfrak{L}$, there are at most $(\dim U)^2-\dim U+1$ distinct eigenvalues for which $\rho(y_i)\neq 0$.
\end{lem}
Now consider the case $N=2$ where $\mathcal{S}_2=\text{Span}_{\mathbb{C}}\left\{d_1(r),d_2(r),d_{12}(r)|r\in\mathbb{Z}^{2}\right\}$. Combining the two lemmas above shows that there exists $K\in\mathbb{N}$ such that both $\partial_1^{m+1} \partial_2^n D_{12} (e_1) = 0$ and $\partial_1^m\partial_2^{n+1}D_{12}(e_2)=0$ for all $m+n> K$. This fact along with the following lemmas will show that $D_{12}(r)$ is a polynomial plus a delta function.
\begin{lem}\label{df=0forallmngeq0}
If $\partial_1^m\partial_2^nf(r)=0$ for all $m,n\geq 0$ then $f(r+ie_1+je_2)=0$ for all $i,j\geq0$.
\end{lem}
\begin{proof}
Using (\ref{higherorderderiv}) it follows by induction that if $\partial_a^mf(r)=0$ for all $m\geq0$ then $f(r+ie_a)=0$ for all $i\geq 0$. 

Suppose $\partial_1^m\partial_2^nf(r)=0$ for all $m,n\geq 0$, and let $g_{n}(r)=\partial_2^nf(r)$ so that by assumption, for each $n\geq 0$, $\partial_1^mg_{n}(r)=0$ for all $m\geq 0$. By the first part of the proof this implies that $g_{n}(r+ie_1)=0$ for all $i\geq 0$. So for any $i\geq0$, $\partial_2^nf(r+ie_1)=0$ for all $n\geq 0$, which implies that $f(r+ie_1+je_s)=0$ for all $i,j\geq 0$.
\end{proof}
For $K+1$ ordered pairs $(x_i,a_i)$, $i=0,\dots,K$ with distinct $x_i$, there exists a unique interpolating polynomial $P(X)$ of degree at most $K$, such that $P(x_i)=a_i$. This can be extended to functions of two variables in the following way.
\begin{lem}\label{deg2pol}
Given $\frac{(K+1)(K+2)}{2}$ triples $(x_i,y_j,a_{ij})$ for $0\leq i+j\leq K$,  with distinct $x_i$ and distinct $y_j$, there exists a unique polynomial $P(X,Y)$ of degree at most $K$ such that $P(x_i,y_j)=a_{ij}$.
\end{lem}
\begin{proof}
For $K=0$ the constant function $P(X,Y)=a_{00}$ is the unique polynomial of degree 0 through $(x_0,y_0,a_{00})$. Proceed by induction on $K$. The univariate case yields a unique polynomial $R(X)$ of degree at most $K$ such that $R(x_i)=a_{i0}$ for all $i\in\{0,\dots,K\}$. For $i\geq0$ and $j\geq1$ let
\begin{equation*}
b_{ij}=\frac{a_{ij}-R(x_i)}{y_j-y_0}.
\end{equation*}
By induction there is a unique interpolating polynomial $Q(X,Y)$ of degree at most $K-1$ such that $Q(x_i,y_j)=b_{ij}$ for the $\frac{K(K+1)}{2}$ triples $(x_i,y_j,b_{ij})$ where $1\leq i+j\leq K$, and $j\geq1$. Polynomial $P(X,Y)=R(X)+(Y-y_0)Q(X,Y)$ is of degree at most $K$ and $P(x_i,y_j)=a_{ij}$ for $0\leq i+j\leq K$. 

Suppose $T(X,Y)$ is a polynomial of degree at most $K$ and $T(x_i,y_i)=a_{ij}$ for $0\leq i+j\leq K$. Since the decomposition $T(X,Y)=F(X)+(Y-y_0)G(X,Y)$ is unique for polynomials $F$ and $G$, it must be that $F(X)=R(X)$ and $G(X,Y)=Q(X,Y)$. Hence $P(X,Y)$ is unique. 
\end{proof}
\begin{lem}\label{equalpolynomials}
Let $S=S_1\times\dots\times S_N\in\mathbb{C}^N$, where each $S_i$ is a set with $K+1$ elements, and let $F$ and $G$ be polynomials of degree at most $K$ in $N$ variables, $X_1,\dots,X_N$, that agree on $S$. Then $F=G$.
\end{lem}
\begin{proof} Use induction on $N$ where the case $N=1$ is well known (the case $N=2$ follows from Lemma \ref{deg2pol}). Let $a\in S_1$ and divide $F(X)$ and $G(X)$ by $(X_1-a)$ to get $F(X)=(X_1-a)P(X)+R(X_2,\dots,X_N)$ and $G(X)=(X_1-a)Q(X)+T(X_2,\dots,X_N)$. Then $R$ and $T$ are of degree at most $K$, and agree on $S_2\times\dots\times S_N$. By induction $R=T$ and so $P(x)=Q(x)$ for all $x\in S'=(S_1\setminus\{a\})\times S_2\times\dots\times S_N$. Then $P=Q$ also by induction, since $P$ and $Q$ are of degree at most $K-1$ and $S'$ contains a cube of size $K$. Therefore $F=G$.
\end{proof}
\begin{lem}\label{polynomialonquadrant}
Let $r\in\mathbb{Z}^2$. Suppose $\partial_1^m\partial_2^nf(r)=0$ for all $m+n> K$, for some $K\in\mathbb{N}$. Let $p(t)$ be the bivariate interpolating polynomial of degree at most $K$ such that $p(r+ie_1+je_2)=f(r+ie_1+je_2)$ for $0\leq i+j\leq K$. Then $f(s)=p(s)$ for all $s_1\geq r_1,s_2\geq r_2$. 
\end{lem}
\begin{proof}
Let $h(t)=f(t)-p(t)$. Then (\ref{mixedpartialderiv}) implies that $\partial_1^m\partial_2^nh(r)=0$ for $m+n\leq K$, because $h(r+ie_1+je_2)=0$ for $i+j\leq K$. When $m+n>K$, $\partial_1^m\partial_2^nf(r)=0$ by assumption and $\partial_1^m\partial_2^np(r)=0$ since it is a polynomial of degree at most $K$. Thus
$\partial_1^m\partial_2^nh(r)=0$ for $m+n\geq 0$ and so by Lemma \ref{df=0forallmngeq0}, $f(r+ie_1+je_2)=p(r+ie_1+je_2)$ for all $i,j\geq 0$.
\end{proof}
\begin{prop}\label{N=2case}
Let $N=2$ and let $J=A_2\otimes U$ be a module in category $\mathcal{J}$. Then $D_{12}(r)$ acts on $U$ by
\begin{equation*}
D_{12}(r)=\sum_{k\in\mathbb{Z}_{\geq 0}^2}\frac{r^k}{k!}P^{(k)}_{12}-\delta_{r,0}P_{12}^{(0)}
\end{equation*}
for all $r\in\mathbb{Z}^2$, where $P^{(k)}_{12}\in\text{\emph{End}}(U)$ does not depend on $r$, and the summation is finite.
\end{prop}
\begin{proof}
It follows from Lemmas \ref{eigenvector} and \ref{finiteeigenvals} that there exists a $K\in\mathbb{N}$ such that $\partial_1^{m+1}\partial_2^nD_{12}(e_1)=0$ and $\partial_1^m\partial_2^{n+1}D_{12}(e_2)=0$ for $n+m>K$. By Lemma \ref{polynomialonquadrant} there exist $\text{End}(U)$-valued polynomials $P_1$ and $P_2$ such that 
\begin{equation*}
P_1(r)=\partial_1D_{12}(r)\text{ and }P_2(s)=\partial_2D_{12}(s),
\end{equation*}
for all $r=r_1e_1+r_2e_2$ and $s=s_1e_1+s_2e_2$ with $r_1,s_2\geq1$ and $r_2,s_1\geq0$. 
Taking polynomial difference antiderivatives $\bar{P}_1(r)$ and $\bar{P}_2 (s)$ respectively, we get
\begin{equation*}
D_{12}(r)=\bar{P}_1(r)+g_1(r_2)\text{ and }D_{12}(s)=\bar{P}_2(s)+g_2(s_1),
\end{equation*}
where $\bar{P}_1$ and $\bar{P}_2$ are polynomials, and $g_1$ and $g_2$ are functions of $r_2$ and $s_1$ respectively. Then $\bar{P}_1(r)+g_1(r_2)=\bar{P}_2(r)+g_2(r_1)$ for $r_1,r_2\geq 1$, so
\begin{equation*}
 g_2(r_1)-g_1(r_2)=\bar{P}_1(r)-\bar{P}_2(r).
\end{equation*}
Taking the $m$th difference derivative in $e_1$ where $m>K$ gives
\begin{equation*}
 \partial_1^m g_2(r_1)=\partial_1^m(\bar{P}_1(r)-\bar{P}_2(r))=0
\end{equation*}
which implies that $g_2$ is a polynomial in $r_1$. Similarly $g_1$ is a polynomial in $r_2$. Thus
$D_{12}(r)=\bar{P}_1(r)+g_1(r_2)$ and $D_{12}(r)=\bar{P}_2(r)+g_2(r_1),$ are polynomials that agree on $\mathcal{R}_1=\{(i,j)\in\mathbb{Z}^2|i,j\geq1\}$, and hence must be equal by Lemma \ref{equalpolynomials}. Therefore $D_{12}(r)$ is an $\text{End}(U)$-valued polynomial $Q_1(r)$ on $\mathcal{R}_1$. It remains to show that $D_{12}(r)$ acts by a polynomial $P(r)$ on all of $\mathbb{Z}^2$ except at the origin.

Let $\mathcal{L}$ be the Lie algebra with basis elements $D_{12}(r)$ for $r\in\mathbb{Z}^2$ and Lie bracket given by \ref{bracket}. Consider the automorphisms $\varphi_1$ and $\varphi_2$ of $\mathcal{L}$, where $\varphi_1(D_{12}(r_1,r_2))=-D_{12}(-r_1,r_2)$, $\varphi_2(D_{12}(r_1,r_2))=-D_{12}(r_1,-r_2)$, and their composition $\varphi_2\circ\varphi_1$ where $\varphi_2\circ\varphi_1(D_{12}(r_1,r_2)) = D_{12}(-r_1,-r_2)$. What was proven for $\mathcal{L}$ is also true for its image under these automorphisms. Thus there are $\text{End}(U)$-valued polynomials $Q_2,Q_3$, and $Q_4$ such that $D_{12}(r)=Q_2(r)$ for $r\in\mathcal{R}_2=\{(-i,j)\in\mathbb{Z}^2|i,j\geq1\}$, $D_{12}(r)=Q_3(r)$ for $r\in\mathcal{R}_3=\{(-i,-j)\in\mathbb{Z}^2|i,j\geq1\}$, and $D_{12}(r)=Q_4(r)$ for $r\in\mathcal{R}_4=\{(i,-j)\in\mathbb{Z}^2|i,j\geq1\}$.

Let $\sigma:\mathcal{L}\rightarrow\mathcal{L}$ be the automorphism defined by $\sigma(D_{12}(r_1,r_2))=-D_{12}(-r_1+r_2,r_2)$. When applied to $\mathcal{R}_1$, this implies the existence of an $\text{End}(U)$-valued polynomial $P$ such that $D_{12}(r)=P(r)$ on the region $\mathcal{R}_5=\{(i,j)\in\mathbb{Z}^2|j\geq1,i\leq j\}$. Lemma \ref{equalpolynomials} may be applied to the intersection of $\mathcal{R}_1$ and $\mathcal{R}_5$ which says that $Q_1=P$. Applied again to the intersection of $\mathcal{R}_2$ and $\mathcal{R}_5$ yields that $Q_2=P$. Thus $D_{12}(r)$ acts by $\text{End}(U)$-valued polynomial $P(r)$ for $r\in\{(i,j)\in\mathbb{Z}^2|j\geq1\}$.

Similar techniques may be applied to connect this region with $\mathcal{R}_3$ and $\mathcal{R}_4$. The result that follows is that $D_{12}(r)$ acts by $\text{End}(U)$-valued polynomial $P(r)$ for $r\in\mathbb{Z}^2\setminus\{(0,0)\}$. 

To indicate that the polynomial obtained above is specific to the operator $D_{12}(r)$, write $P_{12}$ instead of $P$. Decompose $P_{12}(r)$ into powers of $r$ as
\begin{equation*}
P_{12}(r)=\sum_{k\in\mathbb{Z}_{\geq 0}^2}\frac{r^k}{k!}P^{(k)}_{12},
\end{equation*}
where the sum is finite, the $P^{(k)}_{12}\in\text{End}(U)$ do not depend on $r$, and the factor of $k!$ is there for convenience. Note that $P_{12}(0,0)=P_{12}^{(0,0)}$, however $D_{12}(0,0)=0$ by definition and so it must act by zero. To avoid a contradiction at the origin a delta function is added so that
\begin{equation*}
D_{12}(r)=\sum_{k\in\mathbb{Z}_{\geq 0}^2}\frac{r^k}{k!}P^{(k)}_{12}-\delta_{r,0}P_{12}^{(0)},
\end{equation*}
which is now valid for all $r\in\mathbb{Z}^2$.
\end{proof}
\begin{lem}\label{Ngeq2case}
Let $N\geq2$ and $J=A_N\otimes U$ be a module in category $\mathcal{J}$. Then for $a,b\in\{1,\dots,N\}$, $a\neq b$, $D_{ab}(r)$ acts on $U$ by
\begin{equation*}
D_{ab}(r)=\sum_{k\in\mathbb{Z}_{\geq 0}^N}\frac{r^k}{k!}P^{(k)}_{ab}
\end{equation*}
for  $r\in\mathbb{Z}^N$ with $(r_a,r_b)\neq(0,0)$, where the $P^{(k)}_{ab}\in\text{\emph{End}}(U)$ do not depend on $r$, and the summation is finite.
\end{lem}
\begin{proof}
It follows from Proposition \ref{N=2case} that for any $a,b\in\{1,\dots,N\}$, $a\neq b$, the operators $D_{ab}(r_ae_a+r_be_b)$ act by polynomials in $r_a,r_b$, when $(r_a,r_b)\neq(0,0)$, since for these values the delta function vanishes.

The result will be proven by induction on $N$, with the induction hypothesis that $D_{12}\left(r_1e_1+\dots+r_{N-1}e_{N-1}\right)$ has polynomial action for $(r_1,r_2)\neq(0,0)$. For convenience this is stated for $D_{12}(r)$, though it holds for any $D_{ab}(r)$ by a change of indices. The basis of induction $N=2$ follows from Proposition \ref{N=2case}.

Assume until otherwise stated that $r_i\neq0$ for all $i\in\{1,\dots,N\}$. Consider
\begin{multline*}
[D_{1N}(r_{N}e_{N}),D_{12}(r_1e_1+\dots+r_{N-1}e_{N-1})]\\=-r_1r_{N}D_{12}(r_1e_1+\dots+r_{N-1}e_{N-1})+r_1r_{N}D_{12}(r_1e_1+\dots+r_Ne_{N}).
\end{multline*}
Both $D_{1N}(r_{N}e_{N})$, and $D_{12}(r_1e_1+\dots+r_{N-1}e_{N-1})$ have polynomial action by the induction hypothesis. Rearrange to get
\begin{multline*}
r_1r_{N}D_{12}(r_1e_1+\dots+r_Ne_{N})\\
=[D_{1N}(r_{N}e_{N}),D_{12}(r_1e_1+\dots+r_{N-1}e_{N-1})]+r_1r_{N}D_{12}(r_1e_1+\dots+r_{N-1}e_{N-1}).\!
\end{multline*}
The right hand side is a polynomial in $r_1,\dots,r_{N}$ and thus $r_1r_{N}D_{12}(r)=P(r)$, for some $\text{End}(U)$-valued polynomial $P$, and $r=r_1e_1+\dots+r_Ne_{N}$. Symmetry in indices $1$ and $2$ yields that $r_2r_{N}D_{12}(r)=Q(r)$, for some $\text{End}(U)$-valued polynomial $Q$. Thus,
\begin{equation*}
r_2P(r)=r_1Q(r).
\end{equation*}
Unique factorization of a polynomial into irreducible factors implies that $P$ factors as $P(r)=r_1\bar{P}(r)$ and so $r_1r_{N}D_{12}(r)=r_1\bar{P}(r)$.  Since $r_1\neq 0$, division of polynomials gives that  $r_{N}D_{12}(r)=\bar{P}(r)$. Thus, $r_{N}D_{12}(r)$ has polynomial action, or more generally, $r_aD_{bc}(r)$ has polynomial action for $a\neq b,c$. 

Fix $s_N\neq0$ and consider
\begin{align*}
&[D_{1N}(r_{1}e_{1}+s_Ne_N),D_{2N}(r_2e_2+r_3e_3+\dots+(r_N-s_N)e_N)]\\
&=r_{1}(r_N-s_N)D_{2N}(r_2e_2+r_3e_3+\dots+(r_N-s_N)e_N)\\
&\quad-s_Nr_2D_{1N}(r_{1}e_{1}+s_Ne_N)-r_{1}(r_N-s_N)D_{2N}(r_1e_1+\dots+r_{N}e_{N})\\
&\quad-s_N(r_N-s_N)D_{12}(r_1e_1+\dots+r_{N}e_{N})+s_Nr_2D_{1N}(r_1e_1+\dots+r_{N}e_{N}).
\end{align*}
Let $r=r_1e_1+\dots+r_{N}e_{N}$, and isolate the $D_{12}$ term to get
\begin{align*}
&s_N(r_N-s_N)D_{12}(r)=r_{1}(r_N-s_N)D_{2N}(r_2e_2+r_3e_3+\dots+(r_N-s_N)e_N)\\
&-[D_{1N}(r_{1}e_{1}+s_Ne_N),D_{2N}(r_2e_2+r_3e_3+\dots+(r_N-s_N)e_N)]\\
&-s_Nr_2D_{1N}(r_{1}e_{1}+s_Ne_N)-r_{1}(r_N-s_N)D_{2N}(r)+s_Nr_2D_{1N}(r).
\end{align*}
The first three terms on the right hand side have polynomial action by induction. The last two terms are of the form $r_aD_{bc}(r)$, as is $r_ND_{12}(r)$ on the left hand side, and hence these have polynomial action by the previous step. Since both left hand side $s_N(r_N-s_N)D_{12}(r)$, and $s_Nr_ND_{12}(r)$ has polynomial action, so does their difference $-s_N^2D_{12}(r)$. Because $s_N$ is a nonzero constant this implies that $D_{12}(r)$ has polynomial action. Again by considering a change of indices, this proves that $D_{ab}(r)$ has polynomial action on the region $\mathcal{R}_0=\bigcap_i\{r_i\neq0\}$. It remains to show that $D_{ab}(r)$ has polynomial action for $(r_a,r_b)\neq(0,0)$.

Let $r,s\in\mathbb{Z}^N$ where $s$ is constant. Rearranging the bracket formula gives
\begin{equation*}
(s_ar_b-r_as_b)D_{ab}(r)=(s_ar_b-r_as_b)(D_{ab}(s)+D_{ab}(r-s))-[D_{ab}(s),D_{ab}(r-s)].
\end{equation*}
On the right hand side $D_{ab}(s)$ is constant in $r$ and, by what was just shown, $D_{ab}(r-s)$ has polynomial action in the region $\mathcal{R}_{s}=\bigcap_i\{r_i\neq s_i\}$. Thus there is an $\text{End}(U)$-valued polynomial $T$ such that $(s_ar_b-r_as_b)D_{ab}(r)=T(r)$ for $r\in\mathcal{R}_s$. Similarly for $s'\in\mathbb{Z}^N$,  there is a polynomial $T'$ such that $(s_a'r_b-r_as_b')D_{ab}(r)=T'(r)$ for $r\in\mathcal{R}_{s'}$. Then
\begin{equation*}
(s_a'r_b-r_as_b')T(r)=(s_ar_b-r_as_b)T'(r)
\end{equation*}
for $r\in\mathcal{R}_{s}\cap\mathcal{R}_{s'}$, which implies that $(s_ar_b-r_as_b)$ is an irreducible factor of $T(r)$. So $(s_ar_b-r_as_b)D_{ab}(r)=(s_ar_b-r_as_b)\bar{T}(r)$, for polynomial $\bar{T}$, and when $s_ar_b-r_as_b\neq0$, $D_{ab}(r)=\bar{T}(r)$. Thus $D_{ab}(r)$ has polynomial action on the region $\mathcal{R}_s\cap\{s_ar_b\neq s_br_a\}$. The union of these regions is $\bigcup_s(\mathcal{R}_s\cap\{s_ar_b\neq s_br_a\})=\{(r_a,r_b)\neq(0,0)\}$. Since these regions are defined by deleting a finite number of hyperplanes from $\mathbb{Z}^N$, the intersection of any two contains a cube of arbitrary size. So any two polynomials that agree on the intersection must be equal. Therefore $D_{ab}(r)$ acts by an $\text{End}(U)$-valued polynomial $P_{ab}$ on the region $\{(r_a,r_b)\neq(0,0)\}$. The polynomial $P_{ab}$ can be decomposed into a finite sum in powers of $r$ so that
\begin{equation*}
D_{ab}(r)=\sum_{k\in\mathbb{Z}_{\geq 0}^N}\frac{r^k}{k!}P_{ab}^{(k)}.
\end{equation*}
for all $r\in\mathbb{Z}^N$ with $\{(r_a,r_b)\neq(0,0)\}$.
\end{proof}
\begin{prop}\label{Ngeq3case}
Let $N\geq3$ and $J=A_N\otimes U$ a module in category $\mathcal{J}$. Then for $a,b\in\{1,\dots,N\}$, $a\neq b$, $D_{ab}(r)$ acts on $U$ by
\begin{equation*}
D_{ab}(r)=\sum_{k\in\mathbb{Z}_{\geq 0}^N}\frac{r^k}{k!}P^{(k)}_{ab}
\end{equation*}
for all $r\in\mathbb{Z}^N$, where the $P^{(k)}_{ab}\in\text{\emph{End}}(U)$ do not depend on $r$, and the summation is finite. In addition $P_{ab}^{(k)}=0$ when $k_a=k_b=0$.
\end{prop}
\begin{proof}
Since $D_{ab}\left(r\right)=0$ when $r_a=r_b=0$ by definition, it follows from Lemma \ref{Ngeq2case} that $D_{ab}(r)$ may be expressed as
\begin{equation*}
D_{ab}(r)=\sum_{k\in\mathbb{Z}_{\geq 0}^N}\frac{r^k}{k!}P_{ab}^{(k)}-\delta_{r_a,0}\delta_{r_b,0}\sum_{\substack{k\in\mathbb{Z}_{\geq 0}^{N}\\k_a=k_b=0}}\frac{r^k}{k!}P_{ab}^{(k)}
\end{equation*}
which holds for all $r\in\mathbb{Z}^N$. Now substitute the expression above into the equation $r_cD_{ab}(r)+r_aD_{bc}(r)+r_bD_{ca}(r)=0$ on the region $r_i\neq0$ for all $i\in\{1,\dots,N\}$. The terms with delta functions vanish leaving
\begin{equation}\label{P_ab^krelationship}
r_c\sum_{k\in\mathbb{Z}_{\geq 0}^N}\frac{r^k}{k!}P_{ab}^{(k)}+r_a\sum_{k\in\mathbb{Z}_{\geq 0}^N}\frac{r^k}{k!}P_{bc}^{(k)}+r_b\sum_{k\in\mathbb{Z}_{\geq 0}^N}\frac{r^k}{k!}P_{ca}^{(k)}=0.
\end{equation}
Extracting the coefficient of $r_cr^k$ for $k\in\mathbb{Z}_{\geq 0}^N$ with $k_a=k_b=0$ yields that
\begin{equation*}
\frac{1}{k!}P_{ab}^{(k)}=0.
\end{equation*}
This shows that the terms with delta functions are not necessary, and therefore
\begin{equation*}
D_{ab}(r)=\sum_{k\in\mathbb{Z}_{\geq 0}^N}\frac{r^k}{k!}P_{ab}^{(k)}
\end{equation*}
for all $r\in\mathbb{Z}^N$.
\end{proof}
\section{Classification}	
Consider the Lie algebra of derivations of polynomials in $N$ variables,
\begin{equation*}
\text{Der}(\mathbb{C}[x_1,\dots,x_N])=\text{Span}_{\mathbb{C}}\left\{\left.x^k\frac{\partial}{\partial x_a}\right|a\in\{1,\dots,N\},k\in\mathbb{Z}_{\geq0}^N\right\},
\end{equation*}
and its subalgebra consisting of divergence zero elements,
\begin{equation*}
S_N=\text{Span}_{\mathbb{C}}\left\{\left.S_{ab}(k)\right|a,b\in\{1,\dots,N\},k\in\mathbb{Z}_{\geq0}^N\right\},
\end{equation*}
where $S_{ab}(k)=k_bx^{k-e_b}\frac{\partial}{\partial x_a}-k_ax^{k-e_a}\frac{\partial}{\partial x_b}$. Its Lie bracket is given by
\begin{multline*}
[S_{ab}(q),S_{cd}(k)]=q_bk_cS_{ad}(q+k-e_b-e_c)-q_bk_dS_{ac}(q+k-e_b-e_d)\\
-q_ak_cS_{bd}(q+k-e_a-e_c)+q_ak_dS_{bc}(q+k-e_a-e_d).
\end{multline*}
Note that $S_{ab}(e_a)=-\frac{\partial}{\partial x_b}$ and $S_{ab}(e_b)=\frac{\partial}{\partial x_a}$.

For $n\in\mathbb{N}$ let $\mathfrak{L}_n=\text{Span}_{\mathbb{C}}\left\{\left.S_{ab}(k)\right|a,b\in\{1,\dots,N\},|k|=n+2\right\}$ so that $S_N=\bigoplus_{i=-1}^{\infty}\mathfrak{L}_i$. The bracket above  
shows that $[\mathfrak{L}_i,\mathfrak{L}_j]\subset \mathfrak{L}_{i+j}$. 
\begin{lem}\label{slNmodules}
In the grading $S_N=\bigoplus_{i=-1}^{\infty}\mathfrak{L}_i$, the component $\mathfrak{L}_0$ is isomorphic to $\mathfrak{sl}_N$ and each $\mathfrak{L}_i$ is an irreducible $\mathfrak{sl}_N$-module.
\end{lem}
\begin{proof}
To see that $\mathfrak{L}_0$ is isomorphic to $\mathfrak{sl}_N$, identify $x_a\frac{\partial}{\partial x_b}$ with $E_{ab}$ and $x_a\frac{\partial}{\partial x_a}-x_b\frac{\partial}{\partial x_b}$ with elements $E_{aa}-E_{bb}$ of the Cartan subalgebra. Each $\mathfrak{L}_i$ is an $\mathfrak{sl}_N$-module via the adjoint action of $\mathfrak{L}_0$.

By Weyl's Theorem on complete reducibilty and the fact that every finite dimensional simple $\mathfrak{sl}_N$-module is a highest weight module, it suffices to show that each $\mathfrak{L}_i$ has a unique highest weight vector. In other words the goal is to show that for each $i$ there exists a unique (up to scalar) $v\in\mathfrak{L}_i$ such that $\left[x_a\frac{\partial}{\partial x_b},v\right]=0$ for all $a,b$ with $a<b$.

An arbitrary member of $\mathfrak{L}_n$ can be expressed as 
$\sum\limits_{|m|=n} u_m$ with
$u_m = $ 
\break
$\sum_{j=1}^{N}C_jx^{m+e_j}\frac{\partial}{\partial x_j}$ where 
$\sum_{j=1}^{N}C_j(m_j+1)=0$, since it has divergence zero. Since $\left[x_a\frac{\partial}{\partial x_a}-x_b\frac{\partial}{\partial x_b},u_m\right]=(m_a-m_b)u_m$ for all $a,b\in\{1,\dots,N\}$, weight vectors of $\mathfrak{L}_n$ must have the form $u_m$ for some fixed $m$.

Let $u_m$ be a highest weight vector for $\mathfrak{L}_n$. Since $x$ may only have nonnegative exponents, two cases arise; either $m_j=-1$ for a single index $j$ and $m_i\geq 0$ otherwise, or else all entries of $m$ are nonnegative. The former forces all coefficients $C_i$ to be zero except for $C_j$, and hence $u_m=Cx^{k}\frac{\partial}{\partial x_j}$ with $k_j=0$. In the latter $u_m=\sum_{j=1}^{N}C_jx^{m+e_j}\frac{\partial}{\partial x_j}$ with $m_j+1>0$ for each $j$.

Suppose $u_m=Cx^{k}\frac{\partial}{\partial x_j}$ with $k_j=0$. Then $\left[x_a\frac{\partial}{\partial x_b},u_m\right]=Ck_bx^{k+e_a-e_b}\frac{\partial}{\partial x_j}-\delta_{aj}Cx^{k+e_j}\frac{\partial}{\partial x_b}$. Since $1\leq a<b\leq N$ it follows that the only $u_m$ of this form annihilated by all raising operators $x_a\frac{\partial}{\partial x_b}$ is $u_m=C x_1^{n+1}\frac{\partial}{\partial x_N}$.

It remains to show that no vectors of the form $u_m=\sum_{j=1}^{N}C_jx^{m+e_j}\frac{\partial}{\partial x_j}$ with $m_j+1>0$ for each $j$, are highest weight vectors. Suppose for $a<b$ that $\left[x_a\frac{\partial}{\partial x_b},\sum_{j=1}^{N}C_jx^{m+e_j}\frac{\partial}{\partial x_j}\right]=0$. The coefficient of $\frac{\partial}{\partial x_b}$ on the right hand side is $(C_b(m_b+1)-C_a)x^{m+e_a}=0$. Letting $b=N$ and varying $a$ shows that $C_a=C_N(m_N+1)$ for $a=1,\dots,N-1$. Plugging this into the expression for the divergence of $u$ gives
\begin{equation*}
\sum_{j=1}^{N-1}C_N(m_N+1)(m_j+1)+C_N(m_N+1)=C_N(m_N+1)\left(\sum_{j=1}^{N-1}(m_j+1)+1\right),
\end{equation*}
which is zero only when $C_N=0$ since each $(m_j+1)$ was assumed to be positive. Then $C_N=0$ implies $C_a=0$ for $a=1,\dots,N-1$, and thus $u_m=0$.
\end{proof}
The action of the $P_{ab}^{(k)}$ with $|k|>1$ will be shown to define a representation of the subalgebra 
\begin{equation*}
S_N^+=\text{Span}_{\mathbb{C}}\left\{\left.S_{ab}(k)\right|a,b\in\{1,\dots,N\},k\in\mathbb{Z}_{\geq0}^N,|k|>1\right\}.
\end{equation*}
\begin{prop}\label{Prepresentation}
The map $\rho(S_{ab}(k))=P_{ab}^{(k)}\in\text{\emph{End}}(U)$ for $|k|>1$ is a finite dimensional representation of $S_N^+$ on $U$ for $N\geq2$.
\end{prop}
\begin{proof}
Using Lemma \ref{Ngeq2case} the Lie bracket of $D_{ab}(r)$ with $D_{cd}(s)$ with $r, s \neq0$ may be expressed
\begin{equation*}
[D_{ab}(r),D_{cd}(s)]=\sum_{j,k\in\mathbb{Z}_{\geq 0}^N}\frac{r^js^k}{j!k!}[P_{ab}^{(j)},P_{cd}^{(k)}],
\end{equation*}
where the left hand side may be computed
\begin{multline}\label{Pbracket}
(r_as_b-r_bs_a)D_{cd}(s)+(r_cs_d-r_ds_c)D_{ab}(r)+r_bs_cD_{ad}(r+s)\\-r_bs_dD_{ac}(r+s)
-r_as_cD_{bd}(r+s)+r_as_dD_{bc}(r+s)\\
=(r_as_b-r_bs_a)\!\sum_{k\in\mathbb{Z}_{\geq 0}^N}\frac{s^k}{k!}P_{cd}^{(k)}+(r_cs_d-r_ds_c)\!\sum_{k\in\mathbb{Z}_{\geq 0}^N}\frac{r^k}{k!}P_{ab}^{(k)}+r_bs_c\sum_{k\in\mathbb{Z}_{\geq 0}^N}\frac{(r+s)^k}{k!}P_{ad}^{(k)}\\
-r_bs_d\sum_{k\in\mathbb{Z}_{\geq 0}^N}\frac{(r+s)^k}{k!}P_{ac}^{(k)}-r_as_c\sum_{k\in\mathbb{Z}_{\geq 0}^N}\frac{(r+s)^k}{k!}P_{bd}^{(k)}+r_as_d\sum_{k\in\mathbb{Z}_{\geq 0}^N}\frac{(r+s)^k}{k!}P_{bc}^{(k)}.
\end{multline}
Thus $[P_{ab}^{(j)},P_{cd}^{(k)}]$ is obtained by extracting the coefficient of $\frac{r^js^k}{j!k!}$ in the expression above. Then for any $j,k\in\mathbb{Z}_{\geq 0}^N$ with $|j|,|k|>1$, the bracket is given by
\begin{multline}\label{Pracket:j,k>1}
[P_{ab}^{(j)},P_{cd}^{(k)}]=\\
j_bk_cP_{ad}^{(j+k-e_b-e_c)}-j_bk_dP_{ac}^{(j+k-e_b-e_d)}-j_ak_cP_{bd}^{(j+k-e_a-e_c)}+j_ak_dP_{bc}^{(j+k-e_a-e_d)}.
\end{multline}
Note that the expression on the right will differ if either $|j|\leq1$ or $|k|\leq1$. The equation above shows that $\rho(S_{ab}(k))=P_{ab}^{(k)}$ preserves the Lie bracket of $S_N^+$ and is therefore a finite dimensional representation on $U$.
\end{proof}
Since $D_{ab}(r)=-D_{ba}(r)$ it follows that $P_{ab}^{(k)}=-P_{ba}^{(k)}$ for any $k\in\mathbb{Z}_{\geq0}^N$. A linear relationship for the $P_{ab}^{(k)}$ is seen in (\ref{P_ab^krelationship}), and extracting the coefficient on $r^k$ with $k=e_b+e_c$ gives that $P_{ab}^{(e_b)}=P_{ac}^{(e_c)}$. For $N\geq3$, $P_{ab}^{(0)}=0$ and $P_{ab}^{(e_i)}=0$ for $i\neq a,b$.

Consider the Lie algebra spanned by $\left\{\left.P_{ab}^{(k)}\right|a,b\in\{1,\dots,N\},k\in\mathbb{Z}_{\geq0}^N\right\}$. As was noted above, the expression in (\ref{Pracket:j,k>1}) is valid only when $|j|,|k|>1$.  The remaining brackets are obtained by extracting the coefficient of $\frac{r^js^k}{j!k!}$ in (\ref{Pbracket}) for appropriate values of $j$ and $k$. Doing so yields that $[P_{ab}^{(j)},P_{cd}^{(k)}]=0$ when either $j$ or $k$ is zero. If $|j|>1,|k|=1$ or $|j|=1,k>1$ the terms on the right hand side of (\ref{Pbracket}) vanish using the relationship in (\ref{P_ab^krelationship}) when $N\geq3$ or cancel directly when $N=2$. When both $|j|=1$ and $|k|=1$ the right hand side of (\ref{Pbracket}) has only terms $P_{ab}^{(0)}$ (for some $a$ and $b$). In the case $N=2$, $[P_{ab}^{(e_a)},P_{ab}^{(e_b)}]=P_{ab}^{(0)}$, and $[P_{ab}^{(e_a)},P_{ab}^{(e_a)}]=0$, however when $N\geq3$, $P_{ab}^{(0)}=0$, and so $[P_{ab}^{(j)},P_{cd}^{(k)}]=0$ when $|j|=1$ and $|k|=1$.
Thus for $N\geq3$ the subset of elements $P_{ab}^{(e_i)}$ spans an abelian algebra with generators $\left\{\left.P_{12}^{(e_2)},P_{i1}^{(e_1)}\right|i=2,\dots,N\right\}$, and in the case $N=2$ there is a Heisenberg  algebra spanned by $\left\{P_{12}^{(e_1)},P_{12}^{(e_2)},P_{12}^{(0)}\right\}$.

Let $\mathfrak{a}_N=\text{Span}_{\mathbb{C}}\{C_i|i=1,\dots,N\}$ be an $N$ dimensional abelian Lie algebra, and $\mathcal{H}=\text{Span}_{\mathbb{C}}\{X,Y,Z\}$ a three dimensional Heisenberg algebra with central element $Z=[X,Y]$. For $N\geq3$ the map $\rho(C_a)=P_{ab}^{(e_b)}$ is a finite dimensional representation of $\mathfrak{a}_N$ on $U$. When $N=2$ the map $\rho(X)=P_{12}^{(e_2)},\rho(Y)=P_{21}^{(e_1)}$, and $\rho(Z)=P_{12}^{(0)}$ is a finite dimensional representation of $\mathcal{H}$ on $U$. The following theorem considers Lie algebras $S_2^+\oplus\mathcal{H}$ and $S_N^+\oplus\mathfrak{a}_N$. In either case the bracket of $\mathcal{H}$ or $\mathfrak{a}_N$ with $S_N^+$ is zero.

Since $[\rho(S_{ab}(e_a+e_b)),\rho(S_{ab}(ne_a))]=n\rho(S_{ab}(ne_a))$ for $n\geq0$, Lemma \ref{finiteeigenvals} implies that for some $k_0\geq0$, $\rho(S_{ab}(ke_a))$ acts as zero on $U$ for all $k\geq k_0$. The irreducibility in Lemma \ref{slNmodules} ensures that all of $\mathfrak{L}_k$ acts as zero. So for some $k_0\geq0$, $\mathfrak{L}_k$ acts trivially on $U$ for all $k\geq k_0$. 
\begin{thm}\label{classification}
Let $\lambda\in\mathbb{C}^N$ and let $\mathcal{J}_{\lambda}$ be the subcategory of modules in $\mathcal{J}$ supported on $\lambda+\mathbb{Z}^N$. 
\begin{enumerate}
\item[(a)]
For $N=2$ there is an equivalence of categories between the category of finite dimensional modules for $S_2^+\oplus\mathcal{H}$ and $\mathcal{J}_{\lambda}$. This equivalence maps $U$ to $A_2\otimes U$ where $U$ is a finite dimensional module for $S_2^+\oplus\mathcal{H}$. The action of $\mathcal{S}_2$ on $A_2\otimes U$ is given by $d_a(t^{s}\otimes u)=(s_a+\lambda_a)t^s\otimes u$ and for $r\neq0$,
\begin{multline}\label{N=2action}
d_{12}(r)(t^{s}\otimes u)=\left(r_2s_1-r_1s_2\right)t^{r+s}\otimes u\\+t^{r+s}\otimes\sum_{\substack{k\in\mathbb{Z}_{\geq 0}^2\\|k|>1}}\frac{r^k}{k!}\rho(S_{12}(k))u+t^{r+s}\otimes \left(r_1\rho(X)-r_2\rho(Y)+\rho(Z)\right)u.
\end{multline}
\item[(b)]
For $N\geq 3$, there is an equivalence of categories between the category of finite dimensional modules for $S_N^+\oplus\mathfrak{a}_N$ and $\mathcal{J}_{\lambda}$. This equivalence maps $U$ to $A_N\otimes U$ where $U$ is a finite dimensional module for $S_N^+\oplus\mathfrak{a}_N$. The action of $\mathcal{S}_N$ on $A_N\otimes U$ is given by $d_a(t^{s}\otimes u)=(s_a+\lambda_a)t^s\otimes u$ and
\begin{multline}\label{Ngeq3action}
d_{ab}(r)(t^{s}\otimes u)=\left(r_bs_a-r_as_b\right)t^{r+s}\otimes u\\+t^{r+s}\otimes\sum_{\substack{k\in\mathbb{Z}_{\geq 0}^N\\|k|>1}}\frac{r^k}{k!}\rho(S_{ab}(k))u+t^{r+s}\otimes \left(r_b\rho(C_a)-r_a\rho(C_b)\right)u.
\end{multline}
\end{enumerate}
\end{thm}
\begin{proof}
Let $J$ be a module in $\mathcal{J}_\lambda$. As was noted at the end of Section 2, the module $J$ may be identified 
with $A_N\otimes U$ where $U$ is the weight space $J_\lambda$. Then (J1) with (J3) yields $d_a(t^{s}\otimes u)=(s_a+\lambda_a)t^s\otimes u$ for $u\in U$.

Section 3 showed that the action of $d_{ab}(r)\in\mathcal{S}_N$ on $J$ is determined by its restriction to $U$ and is given by an $\text{End}(U)$-valued polynomial in $r$. When $r\neq0$,
\begin{equation*}
d_{ab}(r)(t^s\otimes u)=(r_bs_a-s_br_a)t^{r+s}\otimes u + t^{r+s}\otimes \sum_{k\in\mathbb{Z}_{\geq 0}^N}\frac{r^k}{k!}P^{(k)}_{ab}u,
\end{equation*}
and for $r=0$, $d_{ab}(r)=0$ and thus acts trivially. Proposition \ref{Prepresentation} and the remarks that follow show that $U$ is a finite dimensional $S_2^+\oplus\mathcal{H}$-module when $N=2$ and $U$ is a finite dimensional $S_N^+\oplus\mathfrak{a}_N$-module when $N\geq3$. The actions in (\ref{N=2action}) and (\ref{Ngeq3action}) follow.

Conversely let $U$ be a finite dimensional module for $S_2^+\oplus\mathcal{H}$. Identify the elements of $S_2^+\oplus\mathcal{H}$ with the $P_{12}^{(k)}$ as above and let 
\begin{equation*}
D_{12}(r)=\sum_{k\in\mathbb{Z}_{\geq 0}^2}\frac{r^k}{k!}P_{12}^{(k)}.
\end{equation*}
This sum is finite due to the discussion just before the theorem. The Lie bracket of $S_2^+\oplus\mathcal{H}$ yields the commutator relations for the $D_{12}(r)$ operators via equation (\ref{Pbracket}). The Lie bracket of the $D_{12}(r)$ along with the action of $d_{12}(r)$ given in (\ref{SNactiononJ}) recovers the commutator relations in $\mathcal{S}_2$. Thus $A_2\otimes U$ is an $\mathcal{S}_2$-module.

The fact that a finite dimensional $S_N^+\oplus\mathfrak{a}_N$-module $U$ yields a finite dimensional module $A_N\otimes U$ for $\mathcal{S}_N$ follows in a similar fashion for the $N\geq 3$ case.
\end{proof}
\section{Irreducible Tensor Modules}
This section considers simple modules from category $\mathcal{J}$. Note that in a finite dimensional irreducible representation of the Heisenberg algebra, the central element must act by zero. Hence $P_{12}^{(0)}=0$ in the case $N=2$ and so the Heisenberg algebra $\mathcal{H}$ used above gets replaced with the two dimensional abelian algebra $\mathfrak{a}_2$. A further simplification found in irreducible modules is that the action of $S_N^+$ becomes the action of $\mathfrak{sl}_N$, its degree zero component from the grading in Lemma \ref{slNmodules}. The following will be used to show this (cf. \cite{CS} Lemma 2.4  and \cite{FH} Lemma 9.13). 
\begin{lem}\label{trivialaction}
Let $\mathfrak{g}$ be a finite dimensional Lie algebra over $\mathbb{C}$ with solvable radical $\text{\emph{Rad}}(\mathfrak{g})$. Then $[\mathfrak{g},\text{\emph{Rad}}(\mathfrak{g})]$ acts trivially on any finite dimensional irreducible $\mathfrak{g}$-module.
\end{lem}
As noted above, there exists $k_0$ such that $\mathfrak{L}_k$ acts as zero for all $k\geq k_0$, and so the ideal $I=\bigoplus_{k\geq k_0}\mathfrak{L}_k$ must also act trivially. To apply Lemma \ref{trivialaction} consider the finite dimensional Lie algebra $\mathfrak{g}=S_N^+/I\oplus\mathfrak{a}_N$ and its action on $U$. Since $[\mathfrak{L}_n,\mathfrak{L}_m]\subset\mathfrak{L}_{n+m}$ it follows that $\text{Rad}(\mathfrak{g})=\left(\bigoplus_{n>0}\mathfrak{L}_n\right)/I\oplus\mathfrak{a}_N$ and hence $[\mathfrak{g},\text{Rad}(\mathfrak{g})]=\left(\bigoplus_{n>0}\mathfrak{L}_n\right)/I$ acts trivially. Therefore the ideal $\bigoplus_{n>0}\mathfrak{L}_n$ of $S_N^+\oplus\mathfrak{a}_N$ acts trivially on a simple module from category $\mathcal{J}$.
\begin{thm}
Let $\lambda\in\mathbb{C}^N$ and let $\mathcal{J}_{\lambda}$ be the subcategory of modules in $\mathcal{J}$ supported on $\lambda+\mathbb{Z}^N$. For $N\geq 2$ there is a one-to-one correspondence between the finite dimensional irreducible modules for $\mathfrak{sl}_N\oplus\mathfrak{a}_N$ and the irreducible modules in $\mathcal{J}_{\lambda}$. This correspondence maps a finite dimensional irreducible module $V$  for $\mathfrak{sl}_N\oplus\mathfrak{a}_N$ to $A_N\otimes V$. The action of $\mathcal{S}_N$ on $A_N\otimes V$ is given by $d_a(t^{s}\otimes u)=(s_a+\lambda_a)t^s\otimes u$ and
\begin{multline*}
d_{ab}(r)(t^{s}\otimes u)=\left(r_b(s_a+\mu_a)-r_a(s_b+\mu_b)\right)t^{r+s}\otimes u\\+t^{r+s}\otimes\sum_{\substack{i=1\\i\neq a}}^Nr_ir_b\varphi(E_{ia})u-\sum_{\substack{i=1\\i\neq b}}^Nr_ir_a\varphi(E_{ib})u+r_ar_b\varphi(E_{aa}-E_{bb})u,
\end{multline*}
where $\mu_a,\mu_b\in\mathbb{C}$ are the action of $C_a,C_b\in\mathfrak{a}_N$, and $\varphi$ is a representation of $\mathfrak{sl}_N$.
\end{thm}
\begin{proof}
The correspondence is given in Theorem \ref{classification}. By Lemma \ref{trivialaction} and the discussion above, the ideal $I=
\text{Span}\left\{\left.S_{ab}(j)\right|a,b\in\{1,\dots,N\},|j|>2\right\}$ acts trivially on $V$. Then $S_N^+/I\cong\mathfrak{sl}_N$ and so the action of $\mathcal{L}_0$ in (\ref{Ngeq3action}) is represented by elements of $\mathfrak{sl}_N$ as seen in Lemma \ref{slNmodules}. By Schur's Lemma, elements of $\mathfrak{a}_N$ act by scalars and so $\rho(C_a)$ and $\rho(C_b)$ become  $\mu_a$ and $\mu_b$ respectively in (\ref{Ngeq3action}).
\end{proof}
\section{Acknowledgements}
The first author gratefully acknowledges funding from the Natural Sciences and Engineering Research Council of Canada.


\end{document}